\documentclass[11pt]{article}
\usepackage{amsmath,amssymb,amsfonts,amsthm,amscd}
\usepackage{mathrsfs}
\usepackage{mathtools}

\usepackage{mathptmx}

\usepackage{color}
\usepackage[all]{xy}
\usepackage{graphicx}
\usepackage{fancyhdr}
\usepackage{tikz-cd}
\usepackage{bm}

\usepackage[hidelinks]{hyperref}

\begin{document}
\theoremstyle{definition}
\newtheorem{definition}{Definition}[section]

\newtheorem{notation}[definition]{Notation}
\newtheorem{example}[definition]{Example}
\newtheorem{exercise}[definition]{Exercise}

\theoremstyle{plain}
\newtheorem{theorem}[definition]{Theorem}
\newtheorem{proposition}[definition]{Proposition}
\newtheorem{lemma}[definition]{Lemma}
\newtheorem{corollary}[definition]{Corollary}

\theoremstyle{remark}
\newtheorem{remark}[definition]{Remark}

\def\no{\noindent}
\def\lra{\Leftrightarrow}
\def\loa{\longrightarrow}
\def\tsl{\textstyle}
\def\sul{\sum\limits}
\def\Ra{\Rightarrow}
\def\ra{\rightarrow}
\def\ol{\overline}
\def\ln{\langle}
\def\rn{\rangle}
\def\suse{\subseteq}
\def\spse{\supseteq}
\def\el{\varepsilon}
\def\sta{\stackrel}
\def\Im{{\rm Im}}
\def\im{{\rm im}}
\def\ker{{\rm ker}}
\def\Ker{{\rm Ker}}
\def\Spec{{\rm Spec}}
\def\fkm{{\frak m}}
\def\fkn{{\frak n}}
\def\fkp{{\frak p}}
\def\fkq{{\frak q}}
\def\Max{{\rm Max}}
\def\ann{{\rm ann}}
\def\boml{\bigotimes\limits}
\def\bom{\bigotimes}
\def\cok{{\rm cok}}
\def\Cok{{\rm Coker}}
\def\Ext{{\rm Ext}}
\def\Tor{{\rm Tor}}
\def\gd{{\rm ht}}
\def\pd{{\rm pd}}
\def\id{{\rm id}}
\def\fd{{\rm fd}}
\def\Ass{{\rm Ass}}
\def\rank{{\rm rank}}
\def\FPD{{\rm FPD}}
\def\depth{{\rm depth}}
\def\Supp{{\rm Supp}}
\def\End{{\rm End}}
\def\Hom{{\rm Hom}}
\def\en{{\varepsilon}}
\def\lst{\leqslant}
\def\gst{\geqslant}
\def\t{\tau}
\def\T{\mathfrak{T}}
\def\a{\alpha}
\def\b{\beta}
\def\r{\gamma}
\def\deq{\doteq}
\def\hom{{\rm hom}}
\def\mod{{\rm mod}}
\def\add{{\rm add}}
\def\radical{{\rm rad}}
\def\top{{\rm Top}}
\def\soc{{\rm Soc}}
\def\proj{{\rm proj}}
\def\ind{{\rm ind}}
\def\d{\cdot}
\def\lm{\lambda}
\def\star{\ast}
\def\finish{\hfill$\Box$}
\def\st{\theta}
\def\Id{{\rm Id}}
\def\GV{{\rm GV}}
\def\DW{{\rm DW}}
\def\SM{{\rm SM}}
\def\cap{\rm bigcap}
\newcommand{\Dsg}{\mathrm D_{\mathrm{sg}}}
\newcommand{\Db}{\mathrm D^{\mathrm b}}
\newcommand{\Kb}{\mathrm K^{\mathrm b}}
\newcommand{\G}{\mathcal G}
\newcommand{\rad}{\operatorname{rad}}

{\Large \bf
\begin{center}
Global obstruction property and its applications
\end{center}}

\medskip
\centerline{\textbf{Weiheng Xia}$^*$ }

 \renewcommand{\thefootnote}{\alph{footnote}}
 \setcounter{footnote}{-1} \footnote{ $^*$ Corresponding author.
 Email:xia.weiheng@qq.com}
 \renewcommand{\thefootnote}{\alph{footnote}}
 \setcounter{footnote}{-1} 
 \renewcommand{\thefootnote}{\alph{footnote}}
 \setcounter{footnote}{-1} 

 \begin{abstract}
We introduce the global obstruction property (GOP) for Artin algebras: for every finitely generated module of infinite projective dimension, two distinct syzygies have a common nonprojective indecomposable direct summand. We prove that GOP implies the generalized Auslander--Reiten condition, and hence the Auslander--Reiten conjecture. We give criteria for GOP to hold for pullbacks over semisimple algebras and construct an example which is not syzygy-finite. For each Artin algebra $A$, we construct an algebra $\G(A)$ satisfying GOP by using triangular matrix extensions. We show that $\text{fin.dim}\G(A)=\text{fin.dim} A+2$ and that $A$ and $\G(A)$ are singularly equivalent. Finally, we give a pair of derived equivalent finite-dimensional algebras for which GOP holds on only one side.
 \end{abstract}


\section{Introduction}

The vanishing of extension groups provides a way to detect projectivity, but it is often difficult to determine how few test modules are sufficient. A central question in this direction is the Auslander--Reiten conjecture, formulated in the study of the generalized Nakayama conjecture by Auslander and Reiten~\cite{AR75}. For an Artin algebra $A$, the conjecture asserts that a finitely generated left $A$-module $M$ satisfying
\[
  \Ext_A^i(M,M\oplus A)=0 \quad\text{for all }i>0
  \qquad\Longrightarrow\qquad M\text{ is projective}.
\]
Thus one tests $M$ against itself and the regular module. If $A$ is self-injective, the groups $\Ext_A^i(M,A)$ vanish automatically for $i>0$, and the assertion becomes the self-extension problem usually called the second Tachikawa conjecture; see~\cite{Tachikawa73}. A recent preprint by OpenAI proposes a counterexample to ARC for finite-dimensional algebras~\cite{OAIARC26}. Our results do not depend on the validity of this proposed counterexample. We seek structural conditions under which the above vanishing conditions imply projectivity.

Several established approaches impose additional restrictions either on projective resolutions or on cohomology. Auslander and Reiten proved ARC for algebras whose finitely generated modules have ultimately closed projective resolutions~\cite{AR75}. Christensen and Holm proved it for rings satisfying Auslander's condition on vanishing of cohomology~\cite{CH10}. This condition requires, for each fixed module $M$, a bound on the degrees in which $\Ext_A^i(M,N)$ can be nonzero whenever these groups eventually vanish, with the bound independent of $N$. Other affirmative results depend on the structure of the algebra. For example, Xu proved ARC for symmetric special biserial algebras and for special biserial algebras with radical cube zero~\cite{Xu15}.

The corresponding question over commutative Noetherian rings has also motivated substantial work. Over local complete intersections, the support varieties developed by Avramov and Buchweitz relate eventual self-extension vanishing to finite projective dimension~\cite{AB00}; together with vanishing against the ring, this gives ARC. Construction methods are useful beyond the complete intersection case. Kumashiro studied quotients by powers of ideals generated by regular sequences and obtained classes of non-Gorenstein Cohen--Macaulay rings satisfying ARC~\cite{Kumashiro23}. In particular, if $S$ is a local complete intersection and $Q$ is generated by an $S$-regular sequence of length $c$, then $S/Q^\ell$ satisfies ARC for $1\leq\ell\leq c$~\cite[Corollary 2.9]{Kumashiro23}. Such results motivate the search for algebra constructions under which Ext-vanishing can be controlled explicitly.

A stronger condition allows the vanishing to begin in an arbitrary degree. We say that $A$ satisfies the \emph{generalized Auslander--Reiten condition} (GARC) if, for every $M\in A-\mod$ and every integer $r\geq0$,
\[
  \Ext_A^i(M,M\oplus A)=0 \quad\text{for all }i>r
  \qquad\Longrightarrow\qquad \pd_A M\leq r.
\]
This terminology refers to the validity of the generalized Auslander--Reiten conjecture for the algebra in question. GARC does not hold for every Artin algebra: quantum exterior algebras with suitable parameters admit nonprojective modules with eventually vanishing self-extensions. Such examples originate in work of Schulz~\cite{Schulz86}; explicit module and syzygy calculations can be found in~\cite[Section 4.1]{Erdmann17}. Since these algebras are self-injective, the modules also satisfy the required vanishing against the algebra. We give a sufficient condition for GARC in terms of repeated direct summands of syzygies.

Our starting point is repetition among direct summands of syzygies. The study of ultimately closed projective resolutions goes back to Jans~\cite{Jans61}. Goodearl and Huisgen-Zimmermann subsequently developed the notions of syzygy type and repetition index~\cite{GHZ98}. A module has finite syzygy type when only finitely many isomorphism classes of indecomposable modules occur as direct summands of its syzygies. For finitely generated modules over an Artin algebra, this is equivalent to having an ultimately closed projective resolution~\cite[Comments 2.2(a)]{GHZ98}. The algebra-level condition of syzygy-finiteness is stronger: after a fixed number of syzygies, one finite collection of indecomposable modules suffices for all finitely generated modules. These conditions provide natural sources of recurrence in resolutions.

Repetition also has a role in the finitistic dimension problem. The little finitistic dimension of $A$ is
\[
  \text{fin.dim} A
  =\sup\{\pd_A M\mid M\in A-\mod,\ \pd_A M<\infty\}.
\]
The finitistic dimension conjecture asks whether this number is finite for every Artin algebra; its history includes Bass's work on finitistic dimensions~\cite{Bass60}. A recent preprint by OpenAI proposes a finite-dimensional algebra of infinite little finitistic dimension~\cite{OAIFD26}. Our results concerning finitistic dimensions are independent of this proposed construction. Goodearl and Huisgen-Zimmermann showed that repetition indices of suitable individual modules give bounds on finitistic dimensions~\cite{GHZ98}. Our condition requires only that, for each module of infinite projective dimension, some nonprojective indecomposable module occur as a direct summand of two distinct syzygies. As we show below, this condition does not imply finiteness of the little finitistic dimension.

We call this condition the \emph{global obstruction property}, abbreviated GOP. Precisely, $A$ satisfies GOP if, for every $M\in A-\mod$ with $\pd_A M=\infty$, there exist integers $n\geq0$ and $p>0$ and a nonzero indecomposable nonprojective module $O$ such that
\[
  O\mid\Omega_A^n(M)
  \qquad\text{and}\qquad
  O\mid\Omega_A^{n+p}(M),
\]
where $O\mid X$ means that $O$ is isomorphic to a direct summand of $X$, and $\Omega_A$ is defined using projective covers. The integers and the module $O$ may depend on $M$. In particular, GOP requires neither a uniform recurrence bound nor finite syzygy type for every module. Representation-finite algebras and syzygy-finite algebras satisfy GOP, as do algebras of finite global dimension.

Our first result is that GOP implies GARC, and hence ARC (Theorem 3.2 and Proposition 3.4).

\begin{theorem}
If an Artin algebra $A$ satisfies GOP, then it satisfies GARC, hence it satisfies ARC.
\end{theorem}

We next study constructions that produce algebras satisfying GOP. Pullbacks are especially suitable because their radicals may split into two ideals with zero mixed products. In commutative algebra, Dress and Kr\"amer showed that second syzygies over a fibre product of local rings split into modules over its two factors~\cite{DK75}. Nasseh and Sather-Wagstaff used the structure of such resolutions to prove vanishing results for Tor and Ext, including ARC for every nontrivial fibre product of commutative Noetherian local rings over a common residue field~\cite{NSW17}. Their result does not require ARC for the two factors. These results provide the relevant precedent for our analysis of pullbacks of Artin algebras over a semisimple quotient.

Let $\alpha\colon A\twoheadrightarrow D$ and $\beta\colon B\twoheadrightarrow D$ be surjective homomorphisms, with $D$ semisimple. We consider
\[
  A\times_D B
  =\{(a,b)\in A\oplus B\mid(a)\alpha=(b)\beta\}.
\]
When $\ker\alpha=\rad A$ and $\ker\beta=\rad B$, we call this a \emph{top pullback}. We prove a decomposition of second syzygies and explicit recursion formulas for higher syzygies. These formulas imply that $A\times_D B$ satisfies GOP whenever both $A$ and $B$ do (Proposition 3.13). We also give a criterion that assumes GOP on only one factor, together with restrictions on radicals of projective covers and on a subcategory of modules over the other factor (Proposition 3.11).

\begin{proposition}
Let $\alpha:A\twoheadrightarrow D$ and $\beta:B\twoheadrightarrow D$ be
surjective algebra homomorphisms. Assume that $A$ satisfies GOP,
$\ker\beta=\rad B$, and $D$ is semisimple. Suppose that there is a subset $\mathcal{S}\subset D-\text{simp}$ such that
\begin{enumerate}
\item $\forall S\in D-\text{simp}\setminus\mathcal{S}$, $$\rad_A(P_A(S))\notin A-\proj;$$
\item the subcategory $$\mathfrak{T}_{\mathcal{S}}(B)\coloneqq\{Y\in B-\mod|\top_B(Y)\in\add\mathcal{S}\}$$ is of finite representation type.
\end{enumerate}
Then $\Lambda=A\times_D B$ satisfies GOP.
\end{proposition}

\begin{proposition}
Let $A$ and $B$ be Artin algebras satisfying GOP, let $D$ be semisimple, and let $\alpha:A\twoheadrightarrow D$ and $\beta:B\twoheadrightarrow D$ be surjective algebra homomorphisms. Then $\Lambda=A\times_D B$ satisfies GOP.
\end{proposition}

A useful consequence is obtained when $A$ is representation-finite and every simple $A$-module has infinite projective dimension. If $\beta$ induces $B/\rad B\simeq D$, then $A\times_D B$ satisfies GOP (Corollary 3.12). In particular, for a field $K$, a finite-dimensional local $K$-algebra $B$ with residue field $K$, and any $m\geq2$, the algebra
\[
  K[t]/(t^m)\times_K B
\]
satisfies GOP. Section 4 examines the case $m=4$ and
\[
  B=B_q=K\langle X,Y\rangle/(X^2,Y^2,XY-qYX),
\]
where $-q^{-1}$ has infinite multiplicative order. The resulting algebra has infinite representation type and infinite global dimension, is not syzygy-finite, is not special biserial, and has nonzero radical cube.

Our second construction starts with an arbitrary Artin algebra $A$. Write $\Delta=A/\rad A$ and form the triangular matrix algebras
\[
  C_A=\begin{pmatrix}A&0\\ {}_{\Delta}\Delta_A&\Delta\end{pmatrix},
  \qquad
  \G(A)=\begin{pmatrix}C_A&0\\ {}_{\Delta}\Delta_{C_A}&\Delta\end{pmatrix}.
\]
We call $\G(A)$ the \emph{GOPization} of $A$. Using the description of modules over triangular matrix rings by triples, as in~\cite{HV00}, we compute its syzygies and prove that $\G(A)$ satisfies GOP for every $A$ (Proposition 5.7).

\begin{proposition}
For every Artin algebra $A$, the algebra $\G(A)$ satisfies GOP.
\end{proposition}

The construction also allows an exact comparison of little finitistic dimensions:
\[
  \text{fin.dim}\G(A)=\text{fin.dim} A+2
\]
(Lemma 5.8 and Theorem 5.9).

There is a further connection with singularity categories. We give explicit functors which yield
\[
  \Dsg(A)\simeq\Dsg\bigl(\G(A)\bigr)
\]
as triangulated categories (Theorem 5.11). The construction therefore realizes the singularity category of every Artin algebra as that of an algebra satisfying GOP. If the proposed counterexample in~\cite{OAIARC26} is valid, then our construction yields singularly equivalent finite-dimensional algebras such that one satisfies GARC whereas the other fails ARC. Under this assumption, neither ARC nor GARC is preserved under arbitrary triangle equivalences of singularity categories.

Finally, we examine the behaviour of GOP under equivalences. Wei proved that ARC is invariant under derived equivalence~\cite{Wei12}; derived invariance of GARC was established by Diveris and Purin and by Pan~\cite{DP12,Pan13}. Wei also proved derived invariance of syzygy-finiteness~\cite{Wei18}. For singular equivalences, Chen, Hu, Qin and Wang obtained preservation results for ARC under additional homological hypotheses~\cite{CHQW23}. We therefore consider whether GOP is preserved under these equivalences. GOP is Morita invariant and is preserved by stable equivalences commuting with syzygies. However, using Ladkani's derived equivalences for triangular matrix rings~\cite{Ladkani11}, we construct a pair of derived equivalent finite-dimensional algebras for which GOP holds for exactly one of them (Example 6.2). Consequently, GOP is not a derived invariant. Together with derived invariance of GARC, the example also shows that GARC does not imply GOP.

The paper is organized as follows. Section 2 introduces GOP and records its basic properties. Section 3 proves its implications for ARC and GARC and develops the pullback criteria. Section 4 gives a concrete pullback example. Section 5 constructs GOPizations and studies their finitistic dimensions and singularity categories. Section 6 gives the counterexample to derived invariance. Throughout, all algebras are Artin algebras and all modules are finitely generated left modules unless otherwise specified. Morphisms are composed from left to right, and functors from right to left.

\section{Preliminaries}
We first recall the notions of finite syzygy type and ultimately closed projective resolution.

\begin{definition}
Let ~$A$~ be an Artin algebra and ~$M\in A-\mod$~.
\begin{itemize}
\item The module ~$M$~ has finite syzygy type if\[\add(\{\Omega^j(M)|j\ge 0\})\]contains only finitely many isomorphism classes of indecomposable objects. The algebra $A$ is syzygy-finite if there exists an integer $s\ge0$ such that
\[
  \add\{\Omega_A^s(N)\mid N\in A-\mod\}
\]
contains only finitely many isomorphism classes of indecomposable modules.
\item The module ~$M$~ has an ultimately closed projective resolution if there exist ~$n\ge 0$~ and a decomposition into indecomposable modules \[\Omega^n(M)=\mathop{\oplus}\limits_{i\in I}N_i\]such that ~$\forall N_i,i\in I$~ one has ~$N_i|\Omega^{j_i}(M)$~ for some ~$j_i<n$~.
\end{itemize}
\end{definition}

\begin{definition}\label{GOP}
An Artin algebra $A$ satisfies the global obstruction property if, for every ~$M\in A-\mod$~ with ~$\pd_A M=\infty$~, there exist integers ~$n\ge0$~ and ~$p>0$~, and a nonzero indecomposable nonprojective $A$-module ~$O$~, such that ~$O|\Omega^n_A(M)$~ and ~$O|\Omega^{n+p}_A(M)$~.
\end{definition}
(The integers ~$n,p$~ and the module ~$O$~ may depend on ~$M$~. We write ~$N|M$~ when ~$N$~ is a direct summand of ~$M$~.)\\

The following elementary properties are immediate from the definition.

\begin{proposition}
Let $A$ and $B$ be Artin algebras.
\begin{enumerate}
\item If $A$ and $B$ are Morita equivalent, then $A$ satisfies GOP if and only if $B$ does.
\item The product $A\times B$ satisfies GOP if and only if both $A$ and $B$ satisfy GOP.
\item Every representation-finite algebra satisfies GOP.
\item Every syzygy-finite algebra satisfies GOP.
\item If $\text{gl.dim}A<\infty$, then $A$ satisfies GOP.
\end{enumerate}
\end{proposition}

A stable equivalence that commutes with the syzygy functor preserves GOP. This applies, for example, to stable equivalences of Morita type between finite-dimensional algebras.

\begin{proposition}
Let $\Lambda$ and $\Gamma$ be Artin algebras. Suppose that $F:\Lambda-\underline{\mod}\loa\Gamma-\underline{\mod}$ is a stable equivalence for which there is a
natural isomorphism \[F\Omega_{\Lambda}\cong\Omega_{\Gamma}F.\] Then $\Lambda$ satisfies GOP if and only if $\Gamma$ satisfies GOP.
\end{proposition}

We shall also use the following standard fact.

\begin{lemma}
Let $A$ be an Artin algebra and $M\in A-\mod$. If\[\Ext^{>0}_A(M,M\oplus A)=0,\] then\[\Ext^{>0}_A(\Omega^n(M),\Omega^n(M))=0, \forall n\ge0.\]
\end{lemma}

This is \cite[Lemma~1.4]{AR75}.

\section{GOP, ARC, GARC, and pullback constructions}
\begin{lemma}
Let $A$ be an Artin algebra, ~$M\in A-\mod$~, ~$s>0$~. If a nonzero indecomposable non-projective module $O$ is a direct summand of both $M$ and $\Omega^s_A(M)$, then$$\Ext^s_A(M,M)\neq 0.$$
\end{lemma}

\begin{proof}
Take a minimal projective resolution\[\dots\loa P_s\loa\dots\loa P_0\loa M\loa 0.\]Let $f$ be the composite of a split epimorphism $\pi:\Omega_A^s(M)\twoheadrightarrow O$ and a split monomorphism $i:O\hookrightarrow M$. From the defining short exact sequence for the $s$th syzygy we obtain an exact sequence\[\Hom_A(P_{s-1},M)\overset{\a^*}\loa\Hom_A(\Omega^sM,M)\overset{\st}\loa\Ext^1_A(\Omega^{s-1}M,M)\loa 0.\] We claim that $f\notin\ker\theta$. This gives \[0\neq(f)\st\in\Ext^1_A(\Omega^{s-1}M,M)\cong \Ext^s_A(M,M).\] Suppose that $f\in\ker\theta=\im\alpha^*$, then $f=\alpha h$ for some $h:P_{s-1}\loa M$. Consider the composition $O\overset{i'}\loa\Omega^sM\overset{f}\loa M\overset{\pi'}\loa O$, notice that~$\Id_O=i'f\pi'=(i'\a)(h\pi')$~. Thus~$i'\a:O\loa P_{s-1}$~is a split monomorphism, so~$O\in A-\proj$~, a contradiction.
\end{proof}

The next result is now immediate.

\begin{theorem}\label{GOPTH}
Let $A$ be an Artin algebra satisfying GOP. If $M\in A-\mod$ satisfies\[\Ext^{>0}_A(M,M\oplus A)=0\]then $M\in A-\proj$. Consequently, every algebra satisfying GOP satisfies ARC.
\end{theorem}

In fact, GOP implies GARC. This is not merely a consequence of invariance under derived equivalence; GOP itself is not a derived invariant. We first show that repeated summands can occur at arbitrarily large degree gaps.

\begin{lemma}\label{arbitrarilylonggap}
Let $A$ satisfy GOP and let $X\in A-\mod$. Suppose that \[\Ext^{>0}_A(X,A)=0 \qquad\text{and}\qquad \pd_{A}X=\infty.\] Then for every $L\ge 0$, there exist $s\ge 0,q>L$ and $0\neq O\in\ind A\setminus A-\proj$ such that \[O|\Omega^s_A(X),O|\Omega^{s+q}_A(X).\]
\end{lemma}

\begin{proof}
We write
\[
  ^\bot A
  :=\{M\in A-\mod\mid \Ext_A^i(M,A)=0
      \text{ for all }i>0\}.
\]Recall that, for $Z\in ^\bot A$, then $$Z\notin A-\proj \quad\Longleftrightarrow\quad \pd_{A}Z=\infty.$$ Set $X_0=X$. By GOP, there exist $a_1\ge0$, $p_1>0$, and
$0\ne X_1\in\ind A\setminus A-\proj$ such that $$X_1|\Omega^{a_1}_{A}X_0,\qquad X_1|\Omega^{a_1+p_1}_{A}X_0.$$ Because $^\bot A$ is closed under syzygies and direct summands, ~$X_1\in ^\bot A$~; since $X_1$ is nonprojective, $\pd_A X_1=\infty$. Applying GOP to $X_1$, choose $a_2\ge0$, $p_2>0$, and
$0\ne X_2\in\ind A\setminus A-\proj$ with \[X_2|\Omega^{a_2}_{A}X_1, \qquad X_2|\Omega^{a_2+p_2}_{A}X_1.\] It follows that \[X_2|\Omega^{a_2}_{A}X_1|\Omega^{a_1+a_2}_{A}X_0, \qquad   X_2|\Omega^{a_2+p_2}_{A}X_1|\Omega^{(a_1+a_2)+(p_1+p_2)}_{A}X_0.\] Iterating this construction makes $p_1+\cdots+p_r$ larger than any prescribed
$L$, which proves the claim.
\end{proof}

\begin{proposition}
If an Artin algebra $A$ satisfies GOP, then it satisfies GARC.
\end{proposition}

\begin{proof}
Let $M\in A-\mod$ and suppose that, for some $n\ge 0$,
\[
\Ext^{>n}_A(M,M\oplus A)=0.
\]
We prove that $\pd_{A}M\le n$.

First suppose, for a contradiction, that $\pd_{A}M=\infty$. Set $N\coloneqq\Omega^n_A(M)$. Then $N\in ^\bot A$ and $\pd_AN=\infty$. Set $L=n$. Apply lemma\ref{arbitrarilylonggap} with $L=n$. We obtain $s\ge 0,q>n$ and $0\neq O\in\ind A\setminus A-\proj$ such that
\[
O|\Omega^s_A(N),\qquad O|\Omega^{s+q}_A(N).
\]
Set $Y=\Omega^s_A(N)=\Omega^{n+s}_A(M)$~, so that $O|Y,O|\Omega^q_A(Y)$, then $\Ext^q_A(Y,Y)\neq 0$. On the other hand, dimension shifting, together with $\Ext^{>n}_A(M,A)=0$, gives
\[
\Ext^q_A(Y,Y)\cong\Ext^{q+n+s}_A(M,\Omega^{n+s}_A(M))\cong\Ext^q_A(M,M)=0,
\]
because $q>n$. This contradiction shows that $\pd_A M<\infty$.

Write $d=\pd_A M$. If $d>n$, consider
\[
  0\longrightarrow\Omega_A^d(M)\longrightarrow P_{d-1}
  \longrightarrow\Omega_A^{d-1}(M)\longrightarrow 0.
\]
Since $\Omega_A^d(M)$ is projective and
\[
  \Ext_A^1\bigl(\Omega_A^{d-1}(M),\Omega_A^d(M)\bigr)
  \cong\Ext_A^d\bigl(M,\Omega_A^d(M)\bigr)=0,
\]
the sequence splits. Hence $\Omega_A^{d-1}(M)$ is projective, contradicting the definition of $d$. Therefore $d\le n$.
\end{proof}

We now introduce a special form of Milnor patching. One may think of it as gluing corresponding vertices of two bound quiver algebras. It is also the noncommutative, nonlocal analogue of the fibre product ring.

\begin{definition}
Let $A,B,D$ be Artin algebras, with $D$ semisimple. Suppose that $\a:A\loa D$ and $\b:B\loa D$ are surjective algebra homomorphisms satisfying $$\Ker \a=\radical A,\qquad\Ker\b=\radical B.$$ The algebra $$\Lambda=A\times_DB=\{(a,b)\in A\oplus B|(a)\a=(b)\b\}$$ is called the top pullback algebra of $A$ and $B$.
\end{definition}

The following lemma shows that every pullback along two surjections is isomorphic to a top pullback.

\begin{lemma}\label{toppullbacktrans}
Let $A,B,D$ be Artin algebras, let $D$ be semisimple, and let $\alpha:A\twoheadrightarrow D$ and $\beta:B\twoheadrightarrow D$ be surjective algebra homomorphisms. Then there exist Artin algebras $\widetilde A,\widetilde B,\widetilde D$, with $\widetilde D$ semisimple, and surjective algebra homomorphisms
\[
  \widetilde\alpha:\widetilde A\twoheadrightarrow\widetilde D,
  \qquad
  \widetilde\beta:\widetilde B\twoheadrightarrow\widetilde D,
\]
such that
\[
  \ker\widetilde\alpha=\rad\widetilde A,
  \qquad
  \ker\widetilde\beta=\rad\widetilde B,
\]
and
\[
  A\times_D B\cong
  \widetilde A\times_{\widetilde D}\widetilde B
\]
as algebras.
\end{lemma}

\begin{proof}
Since $\alpha$ and $\beta$ are surjective, $(\rad A)\a\subset \rad D$ and $(\rad B)\b\subset \rad D$. Because $D$ is semisimple, both maps induce surjections $$\a^{'}:A/\rad A\loa D,\qquad \b^{'}:B/\rad B\loa D.$$

Set $K_A\coloneqq\Ker\a^{'}$ and $K_B\coloneqq\Ker\b^{'}$. Then $$A/\rad A\cong D\times K_A,\qquad B/\rad B\cong D\times K_B$$ as algebras. We also notice that $1_{K_A}=e_A$ and $1_{K_B}=e_B$ are central idempotents of the two semisimple quotient algebras $A/\rad A$ and $B/\rad B$. We define maps
\[
\begin{matrix}
  \begin{matrix}
 q_A:A\longrightarrow K_A\\
a\mapsto\overline{a}e_A
\end{matrix}&\begin{matrix}
 q_B:B\longrightarrow K_B\\
b\mapsto\overline{b}e_B
\end{matrix}
\end{matrix}
\]
Since $e_A,e_B$ are central idempotents, we know $q_A,q_B$ are algebra homomorphisms.

We also have the following algebra homomorphisms which are surjective
\[
\begin{matrix}
  \begin{matrix}
 (\a,q_A):A\longrightarrow D\times K_A\\
a\mapsto((a)\a,(a)q_A)
\end{matrix}&\begin{matrix}
 (\b,q_B):B\longrightarrow D\times K_B\\
b\mapsto((b)\b,(b)q_B)
\end{matrix}
\end{matrix}
\]

Set $\overset{\sim}{A}\coloneqq A\times K_B,\overset{\sim}{B}\coloneqq B\times K_A,\overset{\sim}{D}\coloneqq D\times K_A\times K_B$~, consider the following maps
\[
\begin{matrix}
  \begin{matrix}
 \overset{\sim}{\a}:\overset{\sim}{A}\loa \overset{\sim}{D}\\
(a,u)\mapsto((a)\a,(a)q_A,u)
\end{matrix}&\begin{matrix}
 \overset{\sim}{\b}:\overset{\sim}{B}\loa \overset{\sim}{D}\\
(b,v)\mapsto((b)\b,v,(b)q_B)
\end{matrix}
\end{matrix}
\]
$\overset{\sim}{\a},\overset{\sim}{\b}$ are surjective algebra homomorphisms, and it is clear that $\Ker \overset{\sim}{\a}=\radical \overset{\sim}{A},\Ker\overset{\sim}{\b}=\radical \overset{\sim}{B}$.

Define $\overset{\sim}{\Lambda}\coloneqq \overset{\sim}{A}\times_{\overset{\sim}{D}}\overset{\sim}{B}=\{((a,u),(b,v))\in (A\times K_B)\oplus (B\times K_A)|(a,u)\overset{\sim}{\a}=(b,v)\overset{\sim}{\b}\}$~, as one can see from the definition that $(a,u)\overset{\sim}{\a}=(b,v)\overset{\sim}{\b}$ if and only if $((a)\a,(a)q_A,u)=((b)\b,v,(b)q_B)$.

Write $\Lambda=A\times_DB$, consider the following maps
\[
\begin{matrix}
  \begin{matrix}
 \Phi:\Lambda\longrightarrow \overset{\sim}{\Lambda}\\
(a,b)\mapsto((a,(b)q_B),(b,(a)q_A))
\end{matrix}&\begin{matrix}
 \Psi:\overset{\sim}{\Lambda}\longrightarrow \Lambda\\
((a,u),(b,v))\mapsto(a,b)
\end{matrix}
\end{matrix}
\]
Both maps are well defined.  By $(a,b)\in\Lambda$, we know that $(a)\a=(b)\b$, and then $(a,(b)q_B)\overset{\sim}{\a}=((a)\alpha,(a)q_A,(b)q_B)=((b)\b,(a)q_A,(b)q_B)=(b,(a)q_A)\overset{\sim}{\b}$, so $((a,(b)q_B),(b,(a)q_A))\in \overset{\sim}{\Lambda}$. On the other hand, by $((a,u),(b,v))\in \overset{\sim}{\Lambda}$ we know that $(a)\a=(b)\b$, so $(a,b)\in\Lambda$. And since $q_A,q_B$ are algebra homomorphisms, it is clear that $\Phi$ is an algebra homomorphisms.

Notice that
\[
\begin{matrix}
(a,b)\Phi\Psi=((a,(b)q_B),(b,(a)q_A))\Psi=(a,b) \\
((a,u),(b,v))\Psi\Phi=(a,b)\Phi=((a,(b)q_B),(b,(a)q_A))
\end{matrix}
\]
By definition of $\overset{\sim}{\Lambda}$ we know that $(a)q_A=v,(b)q_B=u$~, and then $\Phi\Psi=\Id_{\Lambda},\Psi\Phi=\Id_{\overset{\sim}{\Lambda}}$~. Thus $\Phi$ is an
algebra isomorphism.
\end{proof}

For a top pullback $\Lambda=A\times_D B$, set
\[
  I=\{(x,0)\in\Lambda\mid x\in\rad A\},
  \qquad
  J=\{(0,y)\in\Lambda\mid y\in\rad B\}.
\]

\begin{remark}
The ideals $I$ and $J$ are orthogonal two-sided ideals, and
\[
\begin{gathered}
  \Lambda/J\cong A,
  \qquad \Lambda/I\cong B,
  \qquad IJ=JI=0,\\
  \Lambda/(I\oplus J)\cong D,
  \qquad \rad\Lambda=I\oplus J.
\end{gathered}
\]
Moreover, for every $P\in\Lambda-\proj$,
\[
  \rad_\Lambda P=IP\oplus JP.
\]
\end{remark}

\begin{proof}
The projection $\pi_A:\Lambda\to A$, $(a,b)\mapsto a$, is surjective and has
kernel $J$, so $\Lambda/J\cong A$; similarly $\Lambda/I\cong B$. The map
\[
  p:\Lambda\longrightarrow D,
  \qquad
  (a,b)\longmapsto(a)\alpha=(b)\beta,
\]
has kernel $I\oplus J$, and hence $\Lambda/(I\oplus J)\cong D$. Since $I$ and
$J$ are nilpotent and the last quotient is semisimple,
$\rad\Lambda=I\oplus J$. The remaining assertions follow directly.
\end{proof}

The next lemma prepares the analysis of higher syzygies over $\Lambda$.

\begin{lemma}
For every $M\in\Lambda-\mod$, there is a decomposition
\[
  \Omega_\Lambda^2(M)\cong X\oplus Y,
\]
where $JX=0$ and $IY=0$. Thus $X$ may be regarded as an $A$-module and $Y$ as
a $B$-module.
\end{lemma}
\begin{proof}
Consider an exact sequence
\[
0\loa\Omega^2M\loa P_1\overset{d}\loa P_0\loa M\loa 0.
\]
For $m\in\Omega_\Lambda^2(M)\subseteq\rad P_1=IP_1\oplus JP_1$, write $m=x+y$ with $x\in IP_1$ and $y\in JP_1$. Since $d$ is $\Lambda$-linear,
\[
  (x)d\in I\Omega_\Lambda(M)\subseteq IP_0,
  \qquad
  (y)d\in J\Omega_\Lambda(M)\subseteq JP_0.
\]
The equality $(x)d+(y)d=0$, together with $IP_0\bigcap JP_0=0$, implies $(x)d=(y)d=0$. Hence $x,y\in\Omega_\Lambda^2(M)$. Therefore
\[
  X:=\Omega_\Lambda^2(M)\bigcap IP_1,
  \qquad
  Y:=\Omega_\Lambda^2(M)\bigcap JP_1
\]
give the required direct sum. Finally, $JX=0$ and $IY=0$ because $JI=IJ=0$.
\end{proof}

We next compute the syzygies of $A$-modules and $B$-modules regarded as $\Lambda$-modules via the quotient maps. The following formulas, together with the preceding lemma, describe all syzygies of degree at least two.

\begin{lemma}\label{branchsyzygy}
Let $X\in A-\mod$ and $Y\in B-\mod$, regarded as $\Lambda$-modules. Then
\begin{align*}
  \Omega_\Lambda(X)
  &\cong \Omega_A(X)\oplus
  \rad_B\bigl(P_B(\top_A X)\bigr),\\
  \Omega_\Lambda(Y)
  &\cong \Omega_B(Y)\oplus
  \rad_A\bigl(P_A(\top_B Y)\bigr).
\end{align*}
\end{lemma}

\begin{proof}
We prove the assertion for $X$; the proof for $Y$ is similar.  Let $P\overset{t}\loa X$ be the projective cover of $X$ as a $\Lambda$-module. Consider the diagram with exact rows
\[
\begin{tikzcd}[ampersand replacement=\&]
0\arrow{r} \& JP \arrow{r}{f} \arrow[swap,dashed,hook]{d}{\r} \& P  \arrow{r}{g} \arrow[swap]{d}{\Id_P} \& P/JP\arrow{r} \arrow[dashed,two heads]{d}{\delta} \& 0\\
0\arrow{r} \&\Omega_{\Lambda}(X) \arrow{r}{s} \& P \arrow{r}{t} \&  X\arrow{r}\& 0
\end{tikzcd}
\]
Because $(JP)ft=J(P)ft\subset JX=0$ we know $\exists!\r:JP\hookrightarrow\Omega_{\Lambda}(X)$, thus $\exists!\delta:P/JP\twoheadrightarrow X$. The snake lemma gives
\[
\Ker\delta\cong\Cok\r=\Omega_{\Lambda}(X)/JP.
\]
Since
\[
JP\subset \Omega_{\Lambda}(X)\subset \rad_{\Lambda}P=IP\oplus JP,
\]
we have$\Omega_{\Lambda}(X)=(\Omega_{\Lambda}(X)\bigcap IP)\oplus JP$~. Now we identify $\Omega_{\Lambda}(X)\bigcap IP$ and $JP$.
\begin{enumerate}
\item
~$\Omega_{\Lambda}(X)\bigcap IP$~:\\
Since $\Omega_{\Lambda}(X)\bigcap IP\cong \Omega_{\Lambda}(X)/JP=\Ker\delta$ as $\Lambda$-modules, we consider $\Omega_{\Lambda}(X)/JP$.

As an $A$-module,
\[
  P/JP\cong (\Lambda/J)\otimes_\Lambda P\in A-\proj,
\]
and
\[
  \Omega_\Lambda(X)/JP
  \subseteq (IP\oplus JP)/JP
  \cong I(P/JP)=\rad_A(P/JP).
\]
Thus $P/JP\overset{\delta}\loa X$ is the projective cover of $X$ in $A-\mod$, and $$\Omega_{\Lambda}(X)\bigcap IP\cong \Omega_{\Lambda}(X)/JP\cong \Omega_{A}(X).$$
\item
~$JP$~:\\
Observe that there is a $\Lambda$-module isomorphism:$$JP\cong (IP\oplus JP)/IP=(IP+JP)/IP\cong J(P/IP)=\rad_B(P/IP).$$

On the other hand, there is a $\Lambda$-module isomorphism:$$\frac{P/IP}{J(P/IP)} \cong\frac{P}{IP+JP}=\frac{P}{IP\oplus JP}=\top_{\Lambda}P\cong \top_{\Lambda}X=\frac{X}{\rad_{\Lambda}X}=\frac{X}{IX\oplus JX}=\frac{X}{IX}=\top_AX.$$

And we know $\top_AX$ can be seen as a $D$-module, thus can be seen as a $B$-module.

We also have $\Lambda$-module isomorphism:$$\top_B(P/IP)=\frac{P/IP}{\rad_B(P/IP)}=\frac{P/IP}{J(P/IP)}\cong \top_AX.$$

Since $P/IP\in B-\proj$ and $\top_B(P/IP)\cong \top_AX$, we know that $P/IP\loa \top_AX$ is a projective cover in $B-\mod$, so $JP\cong \rad_B(P_B(\top_AX))$.
\end{enumerate}
Combining the two summands proves the formula.
\end{proof}

Combining the two lemmas above, we have the following recursion formula.

\begin{corollary}\label{recursion}
For every $M\in\Lambda-\mod$ and every $n\ge 2$, one may choose a decomposition
\[
\Omega^n_{\Lambda}(M)\cong X_n\oplus Y_n,
\]
where $X_n\in A-\mod$ and $Y_n\in B-\mod$, such that
\begin{align*}
  X_{n+1}
  &\cong\Omega_A(X_n)
   \oplus\rad_A\bigl(P_A(\top_B Y_n)\bigr),\\
  Y_{n+1}
  &\cong\Omega_B(Y_n)
   \oplus\rad_B\bigl(P_B(\top_A X_n)\bigr).
\end{align*}
\end{corollary}

We now give two ways to construct pullback algebras satisfying GOP.

\begin{proposition}
Let $\alpha:A\twoheadrightarrow D$ and $\beta:B\twoheadrightarrow D$ be
surjective algebra homomorphisms. Assume that $A$ satisfies GOP,
$\ker\beta=\rad B$, and $D$ is semisimple. Suppose that there is a subset $\mathcal{S}\subset D-\text{simp}$ such that
\begin{enumerate}
\item $\forall S\in D-\text{simp}\setminus\mathcal{S}$, $$\rad_A(P_A(S))\notin A-\proj;$$
\item the subcategory $$\mathfrak{T}_{\mathcal{S}}(B)\coloneqq\{Y\in B-\mod|\top_B(Y)\in\add\mathcal{S}\}$$ is of finite representation type.
\end{enumerate}
Then $\Lambda=A\times_D B$ satisfies GOP.
\end{proposition}

\begin{proof}
Use the notation of lemma\ref{toppullbacktrans}. Set $K_A=\ker\alpha'$ and replace
$$
  A\ \text{by }\widetilde A=A,
  \qquad
  B\ \text{by }\widetilde B=B\times K_A,
  \qquad
  D\ \text{by }\widetilde D=D\times K_A.
$$
By lemma\ref{toppullbacktrans}, this does not change the pullback up to algebra isomorphism. Taking
$$
  \widetilde{\mathcal S}
  =\mathcal S\bigcup\ind(K_A),
$$
the two hypotheses remain valid. Since GOP is invariant under algebra isomorphism, we may therefore assume from the beginning that
$$
  \ker\alpha=\rad A,
  \qquad
  \ker\beta=\rad B.
$$

Let $M\in\Lambda-\mod$ with $\pd_{\Lambda}M=\infty$. There is a decomposition $N\coloneqq \Omega_\Lambda^2(M)\cong X_0\oplus Y_0$ where $X_0\in A-\mod,Y_0\in B-\mod$. And $\forall s\ge 0$, we have $$\Omega^s_{\Lambda}(N)\cong X_s\oplus Y_s.$$ There are three cases.

\textbf{Case 1: $\pd_A X_s=\infty$ for some $s$.}

Since $A$ satisfies GOP, there exist $n\ge0$, $p>0$, and $0\ne O\in\ind A\setminus A-\proj$ such that $$O\mid\Omega^n_A(X_s),\qquad O\mid\Omega^{n+p}_A(X_s).$$ The recursion shows that
$$
  O\mid\Omega_\Lambda^{2+s+n}(M),
  \qquad
  O\mid\Omega_\Lambda^{2+s+n+p}(M).
$$ Moreover, $O$ is nonprojective as a $\Lambda$-module(Otherwise, applying $A\otimes_\Lambda-$ would make it projective over $A$). Hence $\Lambda$ satisfies GOP in this case.

\textbf{Case 2: every $X_s$ has finite $A$-projective dimension, and
for infinitely many $s$ one has} $$\bm{\ind(\top_B(Y_s))\bigcap(D-\text{simp}\setminus\mathcal{S})\neq\emptyset}$$

Because $D-\text{simp}\setminus\mathcal S$ is finite, some fixed $S\in D-\text{simp}\setminus\mathcal S$ occurs as a direct summand of $\top_B(Y_s)$ for infinitely many $s$. By hypothesis, $\rad_A(P_A(S))$ is not projective. Choose a nonprojective indecomposable summand $O$ of it. The recursion gives $O\mid X_{s+1}$ for infinitely many $s$, so $O$ occurs in two distinct syzygies of $M$.

\textbf{Case 3: every $X_s$ has finite $A$-projective dimension, and
there is $s_0$ such that $\top_B(Y_s)\in\add\mathcal S$ for all $s\ge s_0$.}

Since $\mathcal T_{\mathcal S}(B)$ is representation-finite, there is a finite set
\[
  \mathcal C=\{C_1,\dots,C_m\}
\]
containing every indecomposable direct summand of $Y_s$ for $s\ge s_0$. At least one non-projective $\Lambda$-module in $\mathcal C$ must occur in two different $Y_s$. Otherwise, after increasing $s_0$, every $Y_s$ would be projective as a $\Lambda$-module. Then
\[
  0=\Omega_\Lambda(Y_s)
  \cong\Omega_B(Y_s)
   \oplus\rad_A\bigl(P_A(\top_B Y_s)\bigr)
\]
for all sufficiently large $s$. Hence the radical term vanishes, and the recursion gives $X_{s+1}\cong\Omega_A(X_s)$. Since $\pd_A X_s<\infty$, one has $X_t=0$ for all sufficiently large $t$. At that stage,
\[
  Y_{t+1}\cong\Omega_B(Y_t)
  \oplus\rad_B\bigl(P_B(\top_A X_t)\bigr)=0,
\]
because $Y_t$ is projective over $B$ and $X_t=0$. This would force all later syzygies to vanish, contradicting $\pd_\Lambda M=\infty$. Thus a fixed indecomposable nonprojective $C_i$ occurs in two distinct syzygies, proving GOP.
\end{proof}

Taking $\mathcal S=\varnothing$ yields a useful special case.

\begin{corollary}\label{Nak1}
Let $\alpha:A\twoheadrightarrow D$ and $\beta:B\twoheadrightarrow D$ be surjective algebra homomorphisms. Suppose that $A$ is representation-finite, every simple $A$-module has infinite projective dimension, $\ker\beta=\rad B$, and $D$ is semisimple. Then $A\times_D B$ satisfies GOP, and hence ARC.
\end{corollary}

The second construction does not require $\ker\beta=\rad B$, but it assumes
GOP on both sides.

\begin{proposition}
Let $A$ and $B$ be Artin algebras satisfying GOP, let $D$ be semisimple, and let $\alpha:A\twoheadrightarrow D$ and $\beta:B\twoheadrightarrow D$ be surjective algebra homomorphisms. Then $\Lambda=A\times_D B$ satisfies GOP.
\end{proposition}

\begin{proof}
By lemma\ref{toppullbacktrans}, we may assume that $\Lambda$ is a top pullback. Let $M\in\Lambda-\mod$ have infinite projective dimension, and write
\[
  \Omega_\Lambda^2(M)=X_0\oplus Y_0,
  \qquad
  \Omega_\Lambda^s\bigl(\Omega_\Lambda^2(M)\bigr)=X_s\oplus Y_s
\]
as in corollary\ref{recursion}. If $\pd_A X_s=\infty$ or $\pd_B Y_s=\infty$ for some $s$, GOP for $A$ or $B$, respectively, gives the required repeated non-projective summand exactly as in Case~1 of proposition\ref{Nak1}.

Assume now that
\[
  \pd_A X_s<\infty
  \quad\text{and}\quad
  \pd_B Y_s<\infty
  \qquad\text{for every }s.
\]
Define finite sets of simple $D$-modules by
\begin{align*}
  \Sigma_A
  &:=\bigl\{S\in D-\text{simp}\mid
       S\mid\top_B(Y_i)\text{ for some }i\ge0\bigr\},\\
  \Sigma_B
  &:=\bigl\{S'\in D-\text{simp}\mid
       S'\mid\top_A(X_j)\text{ for some }j\ge0\bigr\}.
\end{align*}
The recursion and the finite projective dimension assumptions imply

\begin{align*}
  \pd_A\rad_A(P_A(S))&<\infty \qquad(S\in\Sigma_A),\\
  \pd_B\rad_B(P_B(S'))&<\infty \qquad(S'\in\Sigma_B).
\end{align*}

Consequently, every indecomposable summand occurring among all $X_s$ belongs to the finite set of indecomposable summands in the finitely many syzygies of $X_0$ and the modules $\rad_A(P_A(S))$, $S\in\Sigma_A$. The same argument applies to all $Y_s$.

Since $\pd_\Lambda\Omega_\Lambda^2(M)=\infty$, every $X_s\oplus Y_s$ has a non-projective indecomposable $\Lambda$-module as a direct summand. Restriction of scalars along $\Lambda\twoheadrightarrow A$ and $\Lambda\twoheadrightarrow B$ preserves indecomposability and reflects isomorphisms. Hence only finitely many isomorphism classes of possible non-projective summands occur. The pigeonhole principle therefore produces one that occurs in two different syzygies of $M$.
\end{proof}

We record the following description of simple modules over a top pullback.

\begin{proposition}
A simple $\Lambda$-module $S$ is non-projective if and only if, when regarded as a module over one of the quotient algebras, it is a non-projective simple $A$-module or a non-projective simple $B$-module.
\end{proposition}

\begin{proof}
Let $P_\Lambda(S)\twoheadrightarrow S$ be the projective cover. The module $S$ is non-projective precisely when
\[
  \rad_\Lambda P_\Lambda(S)
  =IP_\Lambda(S)\oplus JP_\Lambda(S)\ne0.
\]
By the proof of lemma\ref{branchsyzygy},
\[
\begin{aligned}
  P_\Lambda(S)/JP_\Lambda(S)&\twoheadrightarrow S,\\
  P_\Lambda(S)/IP_\Lambda(S)&\twoheadrightarrow S
\end{aligned}
\]
are the projective covers of $S$ over $A$ and $B$, respectively, and
\begin{align*}
  IP_\Lambda(S)
  &\cong\rad_A\bigl(P_A(S)\bigr),\\
  JP_\Lambda(S)
  &\cong\rad_B\bigl(P_B(S)\bigr).
\end{align*}
The assertion follows.
\end{proof}

In the following situation, the little finitistic dimension is controlled by one branch of the pullback.
\begin{proposition}\label{Nak2}
Assume that every simple $A$-module has infinite projective dimension. If $\text{fin.dim} A<\infty$, then
\[
  \text{fin.dim}\Lambda\le\text{fin.dim} A+2.
\]
\end{proposition}

\begin{proof}
Let $M\in\Lambda-\mod$ have finite projective dimension and put
\[
  N=\Omega_\Lambda^2(M)=X_0\oplus Y_0,
  \qquad
  \Omega_\Lambda^s(N)=X_s\oplus Y_s\quad(s\ge0).
\]
We first show that $Y_s=0$ for every $s\ge0$. If $Y_s\ne0$, then $\top_B(Y_s)$ is a nonzero semisimple $A$-module, and therefore
\[
  \pd_A\top_B(Y_s)=\infty.
\]
It follows that
\[
  \pd_A\rad_A\bigl(P_A(\top_B Y_s)\bigr)=\infty,
\]
and this module is a direct summand of $X_{s+1}$ and hence of $\Omega_\Lambda^{s+1}(N)$.

We use the following observation: if an $A$-module $Z$ has infinite $A$-projective dimension, then it also has infinite $\Lambda$-projective dimension. Otherwise some $\Omega_\Lambda^m(Z)$ would be projective; the iterated formula in lemma\ref{branchsyzygy} shows that $\Omega_A^m(Z)$ is a direct summand of it. Thus $\Omega_A^m(Z)$ would be projective over $\Lambda$, hence also over $A$, a contradiction. Applying this observation to the radical above contradicts $\pd_\Lambda M<\infty$. Therefore $Y_s=0$ for all $s$.

It follows that
\[
  \Omega_\Lambda^s(N)\cong\Omega_A^s(X_0)
  \qquad(s\ge0).
\]
Let $d=\text{fin.dim} A$ and set $e=\pd_A X_0\le d$. Then $\Omega_\Lambda^{e+1}(N)=0$, so $\pd_\Lambda N\le e$ and
\[
  \pd_\Lambda M\le d+2.
\]
Taking the supremum over all modules of finite $\Lambda$-projective dimension proves the result.
\end{proof}

\section{A concrete GOP algebra obtained by top pullback}

We now give a concrete top pullback algebra $\Lambda$ satisfying GOP. It also has the following features:
\begin{enumerate}
\item $\Lambda$ is not commutative;
\item $\Lambda$ is of infinite representation type;
\item $\Lambda$ is not syzygy-finite;
\item $\Lambda$ is not special biserial;
\item $\text{gl.dim}\Lambda=\infty$;
\item $\Lambda$ is not radical cube zero.
\end{enumerate}

\begin{example}\label{eg}
Let $K$ be a field and let $q\in K^\times$ be such that $-q^{-1}$ has infinite multiplicative order(For instance, take $K=\mathbb Q$ and $q=2$). Set
\[
  A=K[t]/(t^4),
  \qquad
  B_q=K\langle X,Y\rangle/(X^2,Y^2,XY-qYX).
\]
Then the top pullback algebra
\begin{align*}
  \Lambda=A\times_K B_q
  &\cong K\langle t,X,Y\rangle/
  \bigl(t^4,X^2,Y^2,XY-qYX,tX,Xt,tY,Yt\bigr)
\end{align*}
satisfies GOP by corollary\ref{Nak1}; hence ARC holds for $\Lambda$.
\end{example}

\begin{remark}
The algebra in example\ref{eg} is not commutative, because $XY=qYX$ and $q\ne1$. Its quiver has one vertex and three loops, so it is not special biserial. Since
\[
  (\rad_\Lambda\Lambda)^3=Kt^3\ne0,
\]
it is not radical cube zero.

Let $N_i=A/(t^i)$ for $1\le i<4$. Then
\[
  N_3\cong\Omega_A(N_1)\mid\Omega_\Lambda(N_1),
  \qquad
  N_1\cong\Omega_A(N_3)\mid\Omega_\Lambda^2(N_1).
\]
Thus the nonzero module $N_1$ occurs repeatedly as a direct summand of its own syzygies, so $\pd_\Lambda N_1=\infty$ and $\text{gl.dim}\Lambda=\infty$.
\end{remark}

To prove that $\Lambda$ is not syzygy-finite, we use the following family of $B_q$-modules.

\begin{lemma}\label{Mr}
For $\lambda\in K^\times$, let $M_\lambda$ be the two-dimensional $B_q$-module with $K$-basis $u,v$ and action
\[
  Xu=Yu=0,
  \qquad
  Xv=u,
  \qquad
  Yv=\lambda u.
\]
Then:
\begin{enumerate}
  \item $M_\lambda\in\ind B_q$ and $M_\lambda$ is non-projective;
  \item $M_\lambda\cong M_\mu$ if and only if $\lambda=\mu$;
  \item
  \[
    \Omega_{B_q}(M_\lambda)\cong M_{(-q^{-1})\lambda},
    \qquad
    \Omega_{B_q}^n(M_\lambda)
      \cong M_{(-q^{-1})^n\lambda}.
  \]
\end{enumerate}
\end{lemma}

\begin{proof}
one has $X^2u=X^2v=Y^2u=Y^2v=0$, and
\[
  (XY-qYX)u=0,
  \qquad
  (XY-qYX)v=0.
\]
Thus $M_\lambda$ is a $B_q$-module. Note also that
\[
  \rad B_q=(X,Y),
  \qquad
  B_q/(X,Y)\cong K.
\]

For part~(1), the socle is
\[
  \soc_{B_q}(M_\lambda)
  =\{m\in M_\lambda\mid Xm=Ym=0\}=Ku.
\]
Indeed, if $m=au+bv$ lies in the socle, then $Xm=bu=0$ and $Ym=\lambda bu=0$, hence $b=0$. A module with one-dimensional socle cannot decompose as a nontrivial direct sum, so $M_\lambda$ is indecomposable. It is non-projective by dimension.

For part~(2), let $M_\mu=Ku'\oplus Kv'$ with the analogous action, and let $f:M_\lambda\to M_\mu$ be an isomorphism. Since $f$ maps socle to socle,
\[
  (u)f=\alpha u'
  \qquad(0\ne\alpha\in K).
\]
Write $(v)f=\beta u'+\gamma v'$ with $\gamma\ne0$. Using $Xv=u$ gives
\[
  \gamma u'=X(\beta u'+\gamma v')=X((v)f)=(Xv)f=(u)f=\alpha u',
\]
so $\alpha=\gamma$. Using $Yv=\lambda u$ gives
\[
  \lambda\alpha u'
  =(\lambda u)f=(Yv)f=Y((v)f)=\gamma\mu u'.
\]
Thus $\lambda=\mu$. The converse is immediate.

For part~(3), the top of $M_\lambda$ is one-dimensional, so the map
\[
  h:B_q\twoheadrightarrow M_\lambda,
  \qquad
  (1)h=v,
\]
is a projective cover. Put $w_\lambda=Y-\lambda X$. Then $(w_\lambda)h=0$, so $B_qw_\lambda\subseteq\ker h$. With the basis $\{1,X,Y,XY\}$ of $B_q$, one computes
\[
  Xw_\lambda=XY=qYX,
  \qquad
  Yw_\lambda=-\lambda YX,
  \qquad
  XYw_\lambda=0.
\]
Hence $B_qw_\lambda$ has basis $\{w_\lambda,YX\}$ and dimension two. Since $\dim_K\ker h=2$, it follows that
\[
  \ker h=B_qw_\lambda=\Omega_{B_q}(M_\lambda).
\]
Set $u'=qYX$ and $v'=w_\lambda$. Then
\[
  Xu'=Yu'=0,
  \qquad
  Xv'=u',
  \qquad
  Yv'=(-q^{-1})\lambda u'.
\]
Therefore $\Omega_{B_q}(M_\lambda)\cong M_{(-q^{-1})\lambda}$, and iteration gives the last formula.
\end{proof}

\begin{remark}
The algebra $\Lambda$ in example\ref{eg} is of infinite representation type and is not syzygy-finite. Indeed, the hypothesis that $-q^{-1}$ has infinite order implies that $K$ is infinite, and every $B_q$-module may be regarded as a $\Lambda$-module. By lemma\ref{Mr}, the modules $M_\lambda$, $\lambda\in K^\times$, give infinitely many pairwise non-isomorphic indecomposable $\Lambda$-modules.

Moreover, lemma\ref{branchsyzygy} gives
\[
  \Omega_\Lambda(M_\lambda)
  \cong\Omega_{B_q}(M_\lambda)
  \oplus\rad_A\bigl(P_A(\top_{B_q}M_\lambda)\bigr).
\]
Thus $M_{(-q^{-1})\lambda}\mid\Omega_\Lambda(M_\lambda)$, and iteration shows that, for every fixed $n$, the class $\Omega_\Lambda^n(\Lambda-\mod)$ contains infinitely many pairwise non-isomorphic indecomposable direct summands. Hence $\Lambda$ is not syzygy-finite.
\end{remark}

\section{GOPization of an algebra}

For an arbitrary Artin algebra $A$, we now construct a related GOP algebra $\G(A)$ by using triangular matrix ring. We begin with the required preliminaries.

Let $A$ and $B$ be Artin $k$-algebras and let ${}_BM_A$ be a $B$--$A$-bimodule which is finitely generated over $k$, with $k$ acting centrally. Set
\[
  R=\begin{pmatrix}
      A&0\\ {}_BM_A&B
    \end{pmatrix}.
\]

\begin{remark}
The following description of left $R$-modules will be used throughout.
\begin{enumerate}
  \item A left $R$-module is equivalent to a triple
  \[
    L=(U,V,\varphi),
    \qquad
    U\in A-\mod,\quad V\in B-\mod,\quad
    \varphi:M\otimes_AU\longrightarrow V,
  \]
  where $\varphi$ is $B$-linear.

  \item As an abelian group, $L=U\oplus V$, with action
  \[
    \begin{pmatrix}a&0\\m&b\end{pmatrix}
    \binom{u}{v}
    =
    \binom{au}{(m\otimes u)\varphi+bv}.
  \]

  \item A morphism
  \[
    (f,g):(U,V,\varphi)\longrightarrow(U',V',\varphi')
  \]
  consists of an $A$-linear map $f:U\to U'$ and a $B$-linear map $g:V\to V'$ satisfying
  \[
    \varphi g=(1_M\otimes f)\varphi'.
  \]

  \item Direct sums are componentwise:
  \[
  (U,V,\varphi)\oplus(U',V',\varphi')
  =\bigl(U\oplus U',V\oplus V',\varphi\boxplus\varphi'\bigr),
  \]
  where
  \[
    \bigl(m\otimes(u,u')\bigr)(\varphi\boxplus\varphi')
    =\bigl((m\otimes u)\varphi,(m\otimes u')\varphi'\bigr).
  \]
  In particular,
  \[
    (U,V,0)\cong(U,0,0)\oplus(0,V,0).
  \]

  \item Every indecomposable projective left $R$-module is isomorphic to either
  \[
    (P,M\otimes_AP,1_{M\otimes_AP})
    \quad\text{or}\quad
    (0,Q,0),
  \]
  where $P$ is an indecomposable projective $A$-module and $Q$ is an indecomposable projective $B$-module.
\end{enumerate}
\end{remark}

The projective cover and syzygy of an arbitrary triple admit the following explicit descriptions. They are the left-module, lower-triangular versions of \cite[Theorem 5.1 and Section 1]{HV00}.

\begin{lemma}\label{normcover}
Let $L=(U,V,\varphi)$. Choose an $A$-projective cover
\[
  \pi:P\twoheadrightarrow U
\]
and a $B$-projective cover
\[
  \overline\rho:Q\twoheadrightarrow\Cok\varphi.
\]
Let $c:V\twoheadrightarrow\Cok\varphi$ be the quotient map, and choose a lift $\rho:Q\to V$ satisfying $\rho c=\overline\rho$. Define
\[
  h:(M\otimes_AP)\oplus Q\longrightarrow V
\]
by
\[
  (z,q)h=\bigl((z)(1_M\otimes\pi)\bigr)\varphi+(q)\rho.
\]
Then
\[
 (\pi,h):\bigl(P,(M\underset{A}{\otimes}P)\oplus Q,\begin{pmatrix}
 \Id_{M\underset{A}{\otimes}P} & 0
 \end{pmatrix}\bigr)\twoheadrightarrow(U,V,\varphi)
\]
is the projective cover of $L$ as an $R$-module.
\end{lemma}

\begin{lemma}\label{normsyzygy}
Retain the notation of lemma\ref{normcover}. Let
\[
  i:\Omega_A(U)\hookrightarrow P,
  \qquad
  V_1=\Ker h.
\]
Then there is an isomorphism of left $R$-modules
\[
  \Omega_R(L)\cong\bigl(\Omega_A(U),V_1,\psi\bigr),
\]
where
\[
  \psi:M\otimes_A\Omega_A(U)\longrightarrow V_1,
  \qquad
  (w)\psi=\bigl((w)(1_M\otimes i),0\bigr).
\]
\end{lemma}

The following two observations will be used repeatedly in our situation. The first compares projective dimensions; the second gives a precise syzygy formula.

\begin{lemma}\label{sp.pd.compare}
Assume that $B$ is semisimple. For every triple
$L=(U,V,\varphi)\in R-\mod$,
\[
  \pd_RL<\infty
  \quad\Longleftrightarrow\quad
  \pd_AU<\infty.
\]
More precisely,
\[
  \pd_AU\le\pd_RL\le\pd_AU+1.
\]
\end{lemma}

\begin{proof}
If $\pd_RL=n<\infty$, take a finite projective resolution of $L$. Every projective $R$-module occurring in it is of the form
\[
  (P,M\otimes_AP,1)\oplus(0,Q,0),
  \qquad P\in A-\proj,\quad Q\in B-\proj.
\]
Taking first components therefore gives a finite $A$-projective resolution of $U$, so
\[
  \pd_AU\le\pd_RL.
\]

On the other hand, use the canonical short exact sequence
\[
0\longrightarrow(0,V,0)
 \xrightarrow{(0,1_V)}L
 \xrightarrow{(1_U,0)}(U,0,0)\longrightarrow 0.
\]
Because $B$ is semisimple, $(0,V,0)$ is projective. Hence
\[
  \pd_RL\le\pd_R(U,0,0).
\]
For each $P\in A-\proj$ there is a short exact sequence
\[
0\longrightarrow(0,M\otimes_AP,0)
 \longrightarrow(P,M\otimes_AP,1)
 \longrightarrow(P,0,0)\longrightarrow 0.
\]
The first two terms are projective $R$-modules, so
$\pd_R(P,0,0)\le1$.

If $d=\pd_AU<\infty$, choose a resolution
\[
0\longrightarrow P_d\longrightarrow P_{d-1}\longrightarrow\cdots
 \longrightarrow P_0\longrightarrow U\longrightarrow 0.
\]
It gives an exact sequence of $R$-modules
\[
0\longrightarrow(P_d,0,0)\longrightarrow\cdots
 \longrightarrow(P_0,0,0)\longrightarrow(U,0,0)\longrightarrow 0.
\]
Since every $(P_i,0,0)$ has $R$-projective dimension at most one, repeated use of the projective-dimension inequalities in short exact sequences yields
\[
  \pd_R(U,0,0)\le d+1.
\]
Thus $\pd_RL\le\pd_AU+1$.

If $\pd_AU=\infty$, the already proved lower bound forces $\pd_RL=\infty$, so the extended inequality is valid in all cases.
\end{proof}

\begin{lemma}\label{sp.syzygy}
Assume that $B$ is semisimple and that $M$, as a right $A$-module, is semisimple. Let $L=(U,V,\varphi)$, and for $j\ge0$ let
\[
  \pi_j:P_j\twoheadrightarrow\Omega_A^j(U)
\]
be the projective cover. Then there exists a $B$-module $V_1$ such that
\[
  \Omega_R(L)\cong
  \bigl(\Omega_A(U),0,0\bigr)\oplus(0,V_1,0),
\]
and, for every $m\ge2$,
\[
  \Omega_R^m(L)\cong
  \bigl(\Omega_A^m(U),0,0\bigr)
  \oplus(0,M\otimes_AP_{m-1},0).
\]
\end{lemma}

\begin{proof}
Let $i_0:\Omega_A(U)\hookrightarrow P_0$ be the kernel inclusion. Since $\pi_0$ is a projective cover,
\[
  (\Omega_AU)i_0\subseteq\rad_A P_0=(\rad A)P_0.
\]
For $m\in M$ and $u\in\Omega_AU$, write $(u)i_0=\sum_j a_jp_j$ with $a_j\in\rad A$. We know that
\[
  m\otimes (u)i_0
  =\sum_jm\otimes a_jp_j
  =\sum_jma_j\otimes p_j=0,
\]
because $M\rad A=0$. Hence $1_M\otimes i_0=0$. In lemma\ref{normsyzygy}, the structure map $\psi$ is therefore zero, and
\[
  \Omega_R(L)\cong
  \bigl(\Omega_A(U),V_1,0\bigr)
  \cong\bigl(\Omega_A(U),0,0\bigr)\oplus(0,V_1,0).
\]
The second summand is projective because $B$ is semisimple.

For later use, if $X\in A-\mod$ and $\pi:P\twoheadrightarrow X$ is its projective cover, then lemma\ref{normcover} gives the projective cover
\[
  (\pi,0):(P,M\otimes_AP,1)\twoheadrightarrow(X,0,0).
\]
Its kernel has zero structure map, so
\[
  \Omega_R(X,0,0)
  \cong\bigl(\Omega_A(X),0,0\bigr)
       \oplus(0,M\otimes_AP,0).
\]
Starting from the first syzygy of $L$, the projective summand $(0,V_1,0)$ disappears after one further syzygy. Applying the last formula to $X=\Omega_A^j(U)$ and proceeding by induction gives, for $m\ge2$,
\[
  \Omega_R^m(L)\cong
  \bigl(\Omega_A^m(U),0,0\bigr)
  \oplus(0,M\otimes_AP_{m-1},0).
\]
\end{proof}

We now define the GOPization of $A$.

\begin{definition}
Let $A$ be an Artin algebra and set
\[
  \Delta=A/\rad A.
\]
Regard ${}_\Delta\Delta_A$ as a bimodule, with right $A$-action induced by the quotient $A\twoheadrightarrow\Delta$, and define
\[
  C_A=
  \begin{pmatrix}
    A&0\\ {}_\Delta\Delta_A&\Delta
  \end{pmatrix}.
\]
The projection
\[
  \varepsilon:C_A\longrightarrow\Delta,
  \qquad
  \begin{pmatrix}a&0\\m&d\end{pmatrix}\longmapsto d,
\]
is a surjective algebra homomorphism. Give ${}_\Delta\Delta_{C_A}$ the right $C_A$-action induced by $\varepsilon$, and define
\[
  \G(A)=
  \begin{pmatrix}
    C_A&0\\ {}_\Delta\Delta_{C_A}&\Delta
  \end{pmatrix}.
\]
We call $\G(A)$ the \emph{GOPization} of $A$.
\end{definition}

\begin{proposition}\label{GOPization}
For every Artin algebra $A$, the algebra $\G(A)$ satisfies GOP.
\end{proposition}

\begin{proof}

We have
\[
  \Delta\rad A=0,
  \qquad
  \Delta\rad C_A=0,
\]
and $\Delta$ is semisimple, so lemma\ref{sp.pd.compare} and lemma\ref{sp.syzygy} are available.

Let $E_1,\dots,E_s$ be the isomorphism classes of indecomposable left $\Delta$-modules. Since $\Delta$ is semisimple, each $E_i$ is simple and projective. Define
\[
  Q_i=(0,E_i,0)\in C_A-\mod.
\]
Then $Q_i$ is an indecomposable projective $C_A$-module. Finally, put
\[
  O_i=(Q_i,0,0)\in\G(A)-\mod.
\]

We first prove that every $O_i$ is indecomposable and non-projective. A $\G(A)$-endomorphism of $O_i$ is necessarily of the form $(s,0)$ with $s\in\End_{C_A}(Q_i)$, and therefore
\[
  \End_{\G(A)}(O_i)
  \cong\End_{C_A}(Q_i)
  \cong\End_\Delta(E_i)
\]
as rings. The last ring is a division ring, so $O_i$ is indecomposable.

To prove every $O_i$ is non-projective, we may firstly compute the left $\Delta$-module $\Delta\otimes_{C_A}Q_i$. More generally, for $V\in\Delta-\mod$, define
\[
  \Theta_V:\Delta\otimes_{C_A}(0,V,0)\longrightarrow V,
  \qquad
  u\otimes\binom0x\longmapsto ux.
\]
This is a homomorphism of left $\Delta$-modules. To check balancing, take $c=\begin{pmatrix}a&0\\m&d\end{pmatrix}\in C_A$. The right action on the first factor is $u\cdot c=ud$, while $c\binom0x=\binom0{dx}$. Hence
\[
  \bigl((u\cdot c)\otimes\tbinom0x\bigr)\Theta_V
  =(ud)x
  =u(dx)
  =\bigl(u\otimes c\tbinom0x\bigr)\Theta_V.
\]
Define
\[
  \Upsilon_V:V\longrightarrow\Delta\otimes_{C_A}(0,V,0),
  \qquad
  x\longmapsto1_\Delta\otimes\binom0x.
\]
If $d\in\Delta$ and $\iota(d)=\begin{pmatrix}0&0\\0&d\end{pmatrix}\in C_A$, then
\begin{align*}
  (dx)\Upsilon_V
  &=1_\Delta\otimes\binom0{dx}
   =1_\Delta\otimes\iota(d)\binom0x\\
  &=(1_\Delta\cdot\iota(d))\otimes\binom0x
   =d\otimes\binom0x
   =d\bigl(1_\Delta\otimes\binom0x\bigr).
\end{align*}
Thus $\Upsilon_V$ is $\Delta$-linear. It is clear that
\[
  \Upsilon_V\Theta_V=1_V,
  \qquad
  \Theta_V\Upsilon_V=1_{\Delta\otimes_{C_A}(0,V,0)}.
\]
Therefore
\[
  \Delta\otimes_{C_A}(0,V,0)\cong V
\]
as left $\Delta$-modules. And typically,
\[
  \Delta\otimes_{C_A}Q_i\cong E_i
\]
as left $\Delta$-modules.

The canonical exact sequence
\[
0\longrightarrow(0,\Delta\otimes_{C_A}Q_i,0)
 \longrightarrow(Q_i,\Delta\otimes_{C_A}Q_i,1)
 \xrightarrow{(1_{Q_i},0)}(Q_i,0,0)
 \longrightarrow 0
\]
is the projective-cover sequence of $O_i$. Its kernel is isomorphic, as a left $\G(A)$-module, to $(0,E_i,0)$, which is nonzero and projective. Consequently,
\[
  \pd_{\G(A)}O_i=1,
\]
so $O_i$ is nonprojective.

It remains to prove GOP. Let
\[
  L=(Z,V,\psi)\in\G(A)-\mod,
  \qquad
  Z=(X,Y,\varphi)\in C_A-\mod,
\]
and suppose that $\pd_{\G(A)}L=\infty$. By lemma\ref{sp.pd.compare},
\[
  \pd_{\G(A)}L=\infty
  \quad\Longleftrightarrow\quad
  \pd_{C_A}Z=\infty
  \quad\Longleftrightarrow\quad
  \pd_AX=\infty.
\]
Let
\[
  \pi_j:P_j\twoheadrightarrow\Omega_A^j(X)
  \qquad(j\ge0)
\]
be projective covers. Since $\pd_AX=\infty$, every $P_j$ is nonzero. By lemma\ref{sp.syzygy}, for $m\ge2$,
\[
  \Omega_{C_A}^m(Z)
  \cong\bigl(\Omega_A^m(X),0,0\bigr)
  \oplus(0,\Delta\otimes_AP_{m-1},0).
\]
As a left $\Delta$-module,
\[
  \Delta\otimes_AP_{m-1}
  \cong P_{m-1}/(\rad A)P_{m-1}\ne0
\]
by Nakayama's lemma. Hence
\[
  \Delta\otimes_AP_{m-1}
  \cong\bigoplus_{i=1}^sE_i^{a_{i,m}}
\]
for some $a_{i,m}\ge0$, not all zero. Thus, for every $m\ge2$, at least one $Q_i$ is a direct summand of $\Omega_{C_A}^m(Z)$. Since there are only finitely many $Q_i$, the pigeonhole principle gives an index $i$ and integers $2\le a<b$ such that
\[
  Q_i\mid\Omega_{C_A}^a(Z),
  \qquad
  Q_i\mid\Omega_{C_A}^b(Z).
\]

Let $P'_j\twoheadrightarrow\Omega_{C_A}^j(Z)$ be projective covers. Applying lemma\ref{sp.syzygy} to the second triangular extension gives
\[
  \Omega_{\G(A)}^n(L)
  \cong\bigl(\Omega_{C_A}^n(Z),0,0\bigr)
  \oplus(0,\Delta\otimes_{C_A}P'_{n-1},0)
  \qquad(n\ge2).
\]
Taking $n=a$ and $n=b$, we obtain
\[
  O_i\mid\Omega_{\G(A)}^a(L),
  \qquad
  O_i\mid\Omega_{\G(A)}^b(L).
\]
The module $O_i$ is indecomposable and non-projective, so $\G(A)$ satisfies GOP.
\end{proof}

We have the following observations.

\begin{lemma}\label{GOP&fdm+1}
For every Artin algebra $A$,
\[
\text{fin.dim}C_A=\text{fin.dim}A+1.
\]
\end{lemma}

\begin{proof}
We first prove the following statement: if $0\ne X\in A-\mod$ and $\pd_AX=d<\infty$, then
\[
  \pd_{C_A}(X,0,0)=d+1.
\]
Set $P'_d=\Omega_A^d(X)$. Then $0\ne P'_d\in A-\proj$. By lemma\ref{sp.syzygy},
\[
  \Omega_{C_A}^d(X,0,0)
  \cong(P'_d,0,0)\oplus P
\]
for some projective $C_A$-module $P$, and
\[
  \Omega_{C_A}(P'_d,0,0)
  \cong(0,\Delta\otimes_AP'_d,0).
\]
The last module is projective over $C_A$. It is nonzero because
\[
  \Delta\otimes_AP'_d
  \cong P'_d/(\rad A)P'_d\ne0
\]
by Nakayama's lemma. Hence $\pd_{C_A}(X,0,0)=d+1$.

Now let $Z=(X,Y,f)\in C_A-\mod$. There is a canonical exact sequence
\[
0\longrightarrow(0,Y,0)
 \xrightarrow{(0,1_Y)}Z
 \xrightarrow{(1_X,0)}(X,0,0)\longrightarrow 0.
\]
The first term is projective. If $\pd_{C_A}Z<\infty$, then lemma\ref{sp.pd.compare} gives $\pd_AX<\infty$, and therefore
\[
  \pd_{C_A}Z
  \le\pd_{C_A}(X,0,0)
  =\pd_AX+1.
\]
This gives $\text{fin.dim} C_A\le\text{fin.dim} A+1$. The modules $(X,0,0)$ constructed above give the reverse inequality.
\end{proof}

Furthermore, we have the following observation. It shows that the finiteness of finitistic dimension is preserved under GOPization.

\begin{theorem}\label{GOP&fdm+2}
For every Artin algebra $A$,
\[
\text{fin.dim}\mathcal{G}(A)=\text{fin.dim}A+2.
\]
\end{theorem}

\begin{proof}
Let $0\ne X\in A-\mod$ and $\pd_AX=d<\infty$. We claim that
\[
  \pd_{\G(A)}\bigl((X,0,0),0,0\bigr)=d+2.
\]
Using the notation in the proof of lemma\ref{GOP&fdm+1}, one has
\[
  \Omega_{C_A}^{d+1}(X,0,0)
  \cong(0,\Delta\otimes_AP'_d,0)\ne0.
\]
Applying lemma\ref{sp.syzygy} to the second extension gives
\[
  \Omega_{\G(A)}^{d+1}\bigl((X,0,0),0,0\bigr)
  \cong\bigl((0,\Delta\otimes_AP'_d,0),0,0\bigr)\oplus Q
\]
for some projective $\G(A)$-module $Q$. Since $(0,\Delta\otimes_AP'_d,0)$ is a projective $C_A$-module, one further syzygy is
\[
  \Omega_{\G(A)}\bigl((0,\Delta\otimes_AP'_d,0),0,0\bigr)\cong(0,\Delta\otimes_{C_A}(0,\Delta\otimes_AP'_d,0),0)\cong(0,\Delta\otimes_AP'_d,0).
\]
Here the last isomorphism is an isomorphism of left $\G(A)$-modules obtained from the explicit left $\Delta$-module isomorphism in the proof of proposition\ref{GOPization}. By $\Delta\otimes_AP'_d\neq 0$, we know $\Omega_{\G(A)}\bigl((0,\Delta\otimes_AP'_d,0),0,0\bigr)\in \G(A)-\proj$, so the claimed equality of projective dimensions follows.

The upper bound is proved exactly as in lemma\ref{GOP&fdm+1}, now using lemma\ref{sp.pd.compare} for the second extension. Hence
\[
  \text{fin.dim}\G(A)=\text{fin.dim} C_A+1=\text{fin.dim} A+2.
\]
\end{proof}

Finally, GOPization preserves the singularity category. Recall that
\[
  \Dsg(A)=\Db(A-\mod)/\Kb(A-\proj).
\]

\begin{lemma}\label{sstrianglering&singularitycat}
Let
\[
  R=\begin{pmatrix}A&0\\{}_BM_A&B\end{pmatrix}
\]
and assume that $B$ is semisimple. Then there is a triangle equivalence
\[
  \Dsg(R)\simeq\Dsg(A).
\]
\end{lemma}

\begin{proof}
Define exact functors
\begin{align*}
  j^*:R-\mod&\longrightarrow A-\mod,&
    j^*(X,Y,\varphi)&=X,\\
  j_*:A-\mod&\longrightarrow R-\mod,&
    j_*(X)&=(X,0,0),\\
  i_*:B-\mod&\longrightarrow R-\mod,&
    i_*(Y)&=(0,Y,0).
\end{align*}
They induce triangle functors on bounded derived categories, denoted by the same symbols.

If $P'$ is a projective $R$-module, then $j^*(P')$ is projective over $A$.
Thus
\[
  j^*\bigl(\Kb(R-\proj)\bigr)\subseteq\Kb(A-\proj),
\]
and $j^*$ induces a triangle functor
\[
  \overline{j^*}:\Dsg(R)\longrightarrow\Dsg(A).
\]

On the other hand, if $P$ is projective over $A$, then
\[
0\longrightarrow(0,M\otimes_AP,0)
 \longrightarrow(P,M\otimes_AP,1)
 \longrightarrow(P,0,0)\longrightarrow 0
\]
shows that $j_*(P)=(P,0,0)$ has a finite projective resolution over $R$; the first two terms are projective because $B$ is semisimple. Hence $j_*$ sends perfect complexes to perfect complexes and induces
\[
  \overline{j_*}:\Dsg(A)\longrightarrow\Dsg(R).
\]
Since $j^*j_*=1_{A-\mod}$, one has
\[
  \overline{j^*}\circ\overline{j_*}=1_{\Dsg(A)}.
\]

It remains to prove the other composite. Let $L^\bullet\in\mathrm C^{\mathrm b}(R-\mod)$, with
\[
  L^n=(X^n,Y^n,\varphi^n),
  \qquad
  d_L^n=(s^n,t^n).
\]
Then $(X^\bullet,s^\bullet)$ and $(Y^\bullet,t^\bullet)$ are bounded complexes and there is a short exact sequence of complexes
\[
0\longrightarrow i_*(Y^\bullet)
 \longrightarrow L^\bullet
 \longrightarrow j_*(X^\bullet)\longrightarrow 0.
\]
It yields a distinguished triangle in $\Db(R-\mod)$,
\[
  i_*(Y^\bullet)\longrightarrow L^\bullet
  \longrightarrow j_*(X^\bullet)\longrightarrow i_*(Y^\bullet)[1].
\]
Since $B$ is semisimple, every term of $i_*(Y^\bullet)$ is a projective $R$-module. Thus $i_*(Y^\bullet)$ is perfect and becomes zero in $\Dsg(R)$. Consequently,
\[
  L^\bullet\cong j_*(X^\bullet)
  =j_*j^*(L^\bullet)
  \qquad\text{in }\Dsg(R).
\]
This is natural in $L^\bullet$, and therefore
\[
  \overline{j_*}\circ\overline{j^*}
  \cong1_{\Dsg(R)}.
\]
The two induced functors are quasi-inverse triangle equivalences.
\end{proof}

Now we point out the following results that can be obtained immediately.

\begin{theorem}\label{GOP&singularitycat}
For every Artin algebra $A$,
\[
  \Dsg(A)\simeq\Dsg(\G(A)).
\]
\end{theorem}
\section{GOP is not invariant under derived equivalence}

We use the following result of Sefi Ladkani; see \cite[Theorem 4.9, Remark 4.10]{Ladkani11}.

\begin{lemma}\label{derivedequ}
Let $A$ be an Artin $k$-algebra of finite global dimension, let $B$ be an Artin $k$-algebra, and let ${}_BM_A$ be a $B$--$A$-bimodule which is finitely generated as a right $A$-module. Assume that $k$ acts centrally on $M$. $D$ denotes the standard duality, then
\[
  \Lambda=
  \begin{pmatrix}A&0\\{}_BM_A&B\end{pmatrix}
  \qquad\text{and}\qquad
  \Gamma=
  \begin{pmatrix}B&0\\{}_AD(M)_B&A\end{pmatrix}
\]
are derived equivalent.
\end{lemma}

We now construct the promised counterexample to derived invariance.

\begin{example}
Let $K$ be a field and let $q\in K^\times$ be such that $\rho=-q^{-1}$ has infinite multiplicative order. Put
\[
  B_q=K\langle X,Y\rangle/(X^2,Y^2,XY-qYX)
\]
and
\[
  C=
  \begin{pmatrix}
    B_q&0\\{}_KK_{B_q}&K
  \end{pmatrix},
  \qquad
  Q=(0,K,0)\in C-\mod,
\]
where the right $B_q$-action on $K$ is induced by $B_q\twoheadrightarrow B_q/\rad B_q\cong K$. Define
\[
  \Lambda=
  \begin{pmatrix}
    K&0\\{}_CQ_K&C
  \end{pmatrix},
  \qquad
  \Gamma=
  \begin{pmatrix}
    C&0\\{}_KD(Q)_C&K
  \end{pmatrix}.
\]
Then $\Lambda$ and $\Gamma$ are derived equivalent, $\Gamma$ satisfies GOP, and $\Lambda$ does not satisfy GOP.
\end{example}

\begin{proof}
Apply lemma\ref{derivedequ} with $A=K$, $B=C$, and $M=Q$. This gives a derived equivalence between $\Lambda$ and $\Gamma$.

Notice that there is an isomorphism of $K$-$C$-bimodules
\[
  D(Q)\cong{}_KK_C.
\]
Hence
\[
  \Gamma\cong\G(B_q).
\]
By proposition\ref{GOPization}, $\Gamma$ satisfies GOP.

It remains to show that $\Lambda$ does not. For $\lambda\in K^\times$, take the module $M_\lambda$ in lemma\ref{Mr}. Applying lemma\ref{normsyzygy} and lemma\ref{sp.syzygy} to $C$ gives, for every $n\ge1$,
\[
  \Omega_C^n(M_\lambda,0,0)
  \cong(M_{\rho^n\lambda},0,0)\oplus(0,K,0).
\]
Now define the $\Lambda$-module
\[
  Z_\lambda=(0,(M_\lambda,0,0),0).
\]
By lemma\ref{normsyzygy}, we know
\[
  \Omega_\Lambda^n(Z_\lambda)\cong\bigl(0,\Omega_C^n(M_\lambda,0,0),0\bigr)\cong\bigl(0,(M_{\rho^n\lambda},0,0),0\bigr)\oplus\bigl(0,(0,K,0),0\bigr)\qquad(n\ge1).
\]
The second summand is projective over $\Lambda$, because $(0,K,0)$ is projective over $C$. For the first summand, the triple description gives
\[
  \End_\Lambda\bigl(0,(M_{\rho^n\lambda},0,0),0\bigr)
  \cong\End_{B_q}(M_{\rho^n\lambda}),
\]
which is local by lemma\ref{Mr}; hence the module is indecomposable. It is non-projective, since a projective $\Lambda$-module with zero first component is of the form $(0,P,0)$ with $P$ projective over $C$, whereas $(M_{\rho^n\lambda},0,0)$ is not projective over $C$.

Because $\rho$ has infinite order, the modules
\[
  M_{\rho^n\lambda}\qquad(n\ge1)
\]
are pairwise nonisomorphic. Thus $Z_\lambda$ has infinite projective dimension, and each of its positive syzygies has exactly one non-projective indecomposable direct summand, with no isomorphism class occurring twice. Therefore $\Lambda$ does not satisfy GOP.
\end{proof}

\medskip
{\bf Acknowledgement.}

 {\footnotesize

 \smallskip
 Weiheng Xia

School of Mathematical Sciences, Capital Normal University,  100048  Beijing, P. R. China;

 {\tt Email: xia.weiheng@qq.com}

}
\end{document}